\documentclass[12pt]{article}
\usepackage{amssymb,amsfonts,amsmath,amsthm}
\usepackage[normalem]{ulem}
\usepackage{pdfsync}
\usepackage{url}
\usepackage{color}
\newcommand{\ol}{\overline}
\newcommand{\witi}{\widetilde}
\newcommand{\fracd}[2]{\frac {\displaystyle #1}{\displaystyle #2 }}
\newcommand{\ee}{{\mathbb E}}
\newcommand{\hh}{{\mathbb H}}
\newcommand{\nn}{{\mathbb N}}
\newcommand{\pp}{{\mathbb P}}
\newcommand{\rr}{{\mathbb R}}
\newcommand{\xx}{{\mathbb X}}
\newcommand{\zz}{{\mathbb Z}}
\newcommand{\cala}{{\mathcal A}}
\newcommand{\calb}{{\mathcal B}}

\newcommand{\cale}{{\mathcal E}}
\newcommand{\calf}{{\mathcal F}}
\newcommand{\calg}{{\mathcal G}}
\newcommand{\call}{{\mathcal L}}
\newcommand{\calp}{{\mathcal P}}
\newcommand{\calr}{{\mathcal R}}
\newcommand{\cals}{{\mathcal S}}

\newcommand{\veps}{\varepsilon}

\newcommand{\beq}{\begin{eqnarray*}}
\newcommand{\feq}{\end{eqnarray*}}
\newcommand{\beqn}{\begin{eqnarray}}
\newcommand{\feqn}{\end{eqnarray}}
\newcommand{\as}{\mbox{\rm a.\,s.}}
\newcommand{\io}{\mbox{\rm i.\,o.}}
\newcommand{\sign}{\mbox{\rm sign}}

\newcommand{\ind}[1]{\mbox{\large \bf 1}_{\{#1\}}}
\newtheorem{theorem}{Theorem}
\makeatletter \@addtoreset{theorem}{section}\makeatother

\newtheorem{lemma}[theorem]{Lemma}

\newtheorem{assume}[theorem]{Assumption}
\newtheorem*{theorem*}{Theorem}
\newtheorem*{corollary*}{Corollary}
\newtheorem{proposition}[theorem]{Proposition}
\newtheorem{corollary}[theorem]{Corollary}
\newtheorem{remark}[theorem]{Remark}

\newtheorem*{cond-b}{Condition~$\mathbf{B}$}
\newtheorem*{cond-c}{Condition~$\mathbf{C}_\kappa$}
\newtheorem*{cond-cl}{Condition~$\mathbf{C(\kappa,L)}$}

\begin{document}
\title{Random walks in a sparse random environment}
\author{Anastasios Matzavinos\thanks{Division of Applied Mathematics, Brown University, Providence, RI 02912, USA; \qquad $\mbox{}$ \qquad $\mbox{}$ \qquad $\mbox{}$\qquad $\mbox{}$ \qquad e-mail: matzavinos@brown.edu}
\and
Alexander~Roitershtein\thanks{Dept. of Mathematics, Iowa State University, Ames, IA 50011, USA; e-mail: roiterst@iastate.edu}
\and
Youngsoo Seol\thanks{Dept. of Mathematics, University of South Florida, Tampa, FL 33620, USA; e-mail: yseol@usf.edu}
}
\date{September 5, 2016; Revised: November 17, 2016}
\maketitle
\begin{abstract}
We introduce random walks in a sparse random environment on $\zz$ and investigate basic asymptotic properties of this model, such as
recurrence-transience, asymptotic speed, and limit theorems in both the transient and recurrent regimes. The new model combines features of
several existing models of random motion in random media and admits a transparent physical interpretation. More specifically, a random walk
in a sparse random environment can be characterized as a ``locally strong" perturbation of a simple random walk by a random potential induced by ``rare impurities,"
which are randomly distributed over the integer lattice. Interestingly, in the critical (recurrent) regime, our model generalizes Sinai's scaling of $(\log n)^2$ for the location of the random walk after $n$ steps to $(\log n)^\alpha,$ where $\alpha>0$ is a parameter determined by the distribution of the distance between two successive impurities. Similar scaling factors have appeared in the literature in different contexts and have been  discussed in \cite{alpha} and \cite{preds}.
\end{abstract}

\section{Introduction: The model}\label{intro}

We start with a general description of one-dimensional random walks in a random environment.
Let $\Omega=(0,1)^\zz$ and let $\calf$ be the Borel $\sigma-$algebra of subsets of the product space $\Omega.$
A \emph{random environment} is a random element $\omega=(\omega_n)_{n\in\zz}$ of the measurable space $(\Omega,\calf).$
The environment determines the transition kernel of the underlying random walk. Namely,
a \emph{random walk in a random environment} $\omega=(\omega_n)_{n\in \zz} \in \Omega$ is a Markov chain $(X_n)_{n\geq 0}$ on $\zz$, the  transition kernel of which is given by
\beq
P_\omega(X_{n+1}=j|X_n=i)=
\left\{
\begin{array}{ll}
\omega_i&\mbox{if}~j=i+1\\
1-\omega_i&\mbox{if}~j=i-1\\
0&\mbox{otherwise}.
\end{array}
\right.
\feq
The probability distribution of this Markov chain, starting at the initial state $X_0=x,$ $x\in\zz,$ is denoted by $P_{x,\omega}$
and is referred to as the \emph{quenched law} of the random walk.
\par
We denote the probability distribution of $\omega$ in $(\Omega,\calf)$ by $P,$ and we let $E_P$ denote the corresponding expectation operator. That is,
for a measurable function $f(\omega)$ of the environment $\omega$, we have $E_P(f)=\int_\Omega f(\omega)dP(\omega).$
Let $\calg$ be the cylinder $\sigma$-algebra on $\zz^n.$ A \emph{random walk in a random environment} (RWRE) associated with
$P$ is a process $(X,\omega)$ on the measurable space $(\Omega \times \zz^\nn, \calf \otimes \calg)$
equipped with the {\em annealed} probability law $\pp=P \otimes P_\omega$, which is defined by
\beq
\pp_x(F \times G)=\int_F P_{x,\omega}(G) P(d\omega),~~~ F \in \calf,~ G \in \calg.
\feq
The expectations under the laws $P_{x,\omega}$ and $\pp_x$ are denoted by $E_{x,\omega}$ and $\ee_x,$ respectively. We will usually omit
the index $0$ when $x=0,$ which is to say that we will write $P_\omega,$ $E_\omega,$ $\pp,$ and $\ee$ for $P_{0,\omega},$ $E_{0,\omega},$ $\pp_0,$ and $\ee_0,$ respectively.
Notice that, since the process ``learns" the environment according to the Bayes rule as time progresses, $X=(X_n)_{n\geq 0}$
is not, in general, a Markov chain under the annealed measure $\pp.$
\par
We now describe in detail the specific model of random environment that we consider in this paper. Let $(\lambda_k,d_k)_{k\in\zz}$ be a stationary and ergodic sequence of pairs,
such that $\lambda_k\in (0,1)$ and $d_k\in\nn.$ Throughout the paper we denote by $P$ the joint law of the sequence of pairs $(\lambda_k,d_k)_{k\in\zz}.$ For $n\in\zz$ let
\beq
a_n=\left\{
\begin{array}{ll}
\sum_{k=1}^n d_k&\mbox{if}~n>0\\
0&\mbox{if}~n=0\\
-\sum_{k=-n}^{0} d_k&\mbox{if}~n<0.
\end{array}
\right.
\feq
The random variables $a_n$ serve as locations of random impurities in the, otherwise homogeneous, random medium.
More precisely, the {\em sparse random environment} $\omega$ associated with $P$ is defined by
\beqn
\label{senv}
\omega_n=
\left\{
\begin{array}{ll}
\lambda_k&\mbox{if}~n=a_k~\mbox{for some}~k\in\zz,\\
1/2&\mbox{otherwise}.
\end{array}
\right.
\feqn
For future reference, we also define
\beqn
\label{rhoxi}
\rho_n=\frac{1-\omega_n}{\omega_n} \qquad \mbox{and}\qquad \xi_n=\frac{1-\lambda_n}{\lambda_n},\qquad n\in\zz.
\feqn
We refer to the random walk in the environment defined by \eqref{senv} as a {\em random walk in a sparse random environment} (RWSRE).
The primary focus of this paper is to illuminate the dependence of the basic properties of RWSREs  on the distribution of the
sequence $(d_n)_{n\in\zz},$ and compare the dynamics of RWSREs to the dynamics of the classical RWRE, which corresponds to the special case $d_0=1$ $\as$
\par
In the classical RWRE model, $\omega$ is a stationary and ergodic sequence under $P$ \cite{notes}.
It is known (see \cite{kks,sinair,zsurvey,notes} and, for instance, \cite{sinai-a, spect-sinai, chelo} and \cite{sponsors, spotential,jon} for some
recent advances in the recurrent and transient cases, respectively) that asymptotic results for one-dimensional RWREs can usually be stated in terms of certain averages of
functions of $\rho_0$ and explained by means of typical ``landscape features" (such as \emph{traps} and \emph{valleys}) of
the {\em random potential} $(R_n)_{n\in\zz},$ which is associated with the random environment as follows:
\beqn
\label{potential}
R_n =
\begin{cases}
\phantom{-}\sum\limits_{k=1}^n \log\rho_k     & \text{if } n>0,\smallskip\\
-\sum\limits_{k=0}^{|n-1|} \log\rho_{-k} & \text{if } n<0.
\end{cases}
\feqn
We remark that interpreting a RWRE as a random walk in the random potential \eqref{potential} serves to provide heuristic explanations to most results about RWREs, including those discussed in this paper.
\par
In our model, the sparse environment $\omega_n$ is defined as a simple functional of the \emph{marked point process} $(a_n,\lambda_n)_{n\in\zz},$
and it is in general non-stationary. However, it is well known that if $E_p(d_0)<\infty,$ the underlying probability space can be enlarged to include a random variable $M$, such that the random shift $(a_{n+M},\lambda_{n+M})_{n\in\zz}$ of the sequence $(a_n,\lambda_n)_{n\in\zz}$ is stationary and ergodic. Therefore, one should expect that if $E_P(d_0)<\infty$, basic zero-one laws, such as recurrence-transience, existence of asymptotic speed, and ballisticity, are similar to the corresponding features of the RWRE associated with the random environment $\omega=(\lambda_n)_{n\in\zz}.$
\par
However, an analogous claim about the similarity of limit theorems for random walks in the environments $\lambda$ and $\omega$ is less
obvious. Indeed, the dependence structure associated with the environment plays a crucial role in such theorems, and it is clearly not preserved under the transformation $\lambda\mapsto \omega.$ We remark, for instance, that the transformation of an i.i.d. environment yields a stationary and ergodic transformation of a Markov chain. See Sections~\ref{dual} and~\ref{stable1} for a more detailed discussion, and Section 3.5 for the case $E(d_0)=\infty$.
\par
In the continuous setting, a model which is closely related to the RWSREs discussed in this paper is  the multi-skewed Brownian motion introduced in \cite{multiskew}.
A direct discrete-time analogue of the multi-skewed Brownian motion is a \emph{multi-skewed random walk}, which can be introduced
as a quenched variant of our model when $(\lambda_n)_{n\in\mathbb{Z}}$ is a certain deterministic sequence of constants. In accordance to the physical motivation of the model in \cite{multiskew}, the author refers to the marked sites (i.e., the elements of $\cala$ in the author's notation) as interfaces, while the long stretches of ``regular" sites between interfaces are referred to by the author as layers.
\par
We remark that certain random environments that consist of alternating stretches of sites of two different types, and induce sub-linear growth rates
on the underlying random processes, have been considered in \cite{fixede, preds} and, in a slightly different context, in \cite{contact}.
The overlap between results and proof methods in this work and in \cite{fixede,preds} is minimal, and it is discussed in more detail in Section~\ref{main}.
\par
Somewhat related to our work is the study of \cite{perturbation, perturbation1}, where it is shown that an $X_n\sim (\log n)^\alpha$
asymptotic behavior of the random walk can occur under a perturbation of an i.i.d. recurrent environment $(\omega_n)_{n\in\zz}$ of the form $\omega^{\mbox{\rm \tiny new}}_n=\omega_n+f_n,$ where $E_P\bigl(\log \frac{1-\omega_0}{\omega_0}\bigr)=0$ and $f_n$ converges to zero in probability as $|n|\to \infty.$
A heuristic explanation of this phenomenon can be provided by considering that both recurrent and transient random walks ``spread out," and  hence, for a large class of perturbations $f_n,$ a typical landscape of the environment as viewed from the particle at time $n$ can be identified. Moreover, it can be shown that this typical landscape is a dominant factor in determining the asymptotic behavior of the random walk as $n$ goes to infinity.
\par
The rest of the paper is organized as follows. In Section \ref{dual}, we introduce a notion of a dual stationary environment along with the Palm dualities that are used in the proof of our results. In Section \ref{main} we state and discuss our main results for the asymptotic behavior of RWSREs. These include recurrence and transience criteria, asymptotic speed, stable laws, and a Sinai-type result for RWSREs. Finally, in Section \ref{proofs} we present the proofs of our results.

\section{Dual stationary environment}
\label{dual}
The sparse environment introduced in Section~\ref{intro} is in general a non-stationary sequence. The aim of this section is to introduce a \emph{dual} stationary environment
and relate the properties of the RWSRE to the corresponding properties of the RWRE in the dual environment.
If $E_p(d_0)<\infty,$ then the underlying probability space can be enlarged to include a non-negative random variable $M,$ such that the random shift $(\omega_{n-M})_{n\in\zz}$ of the environment $(\omega_n)_{n\in\zz}$ is stationary and ergodic. Furthermore, the distribution of the sparse environment turns out to be the distribution of its stationary version conditioned on the event $0\in\cala.$
\par
In contrast to the usual RWRE, $\omega=(\omega_n)_{n\in\zz}$ is in general a non-stationary sequence in the RWSRE model.
In fact, $\omega$ is \emph{cycle-stationary} under $\pp,$ namely
\beq
\theta^{d_n}\omega=_D \omega \qquad \mbox{under}~\pp\quad \mbox{for all}~n\in\zz,
\feq
where $X=_D Y $ means that the distributions of the random variables $X$ and $Y$ coincide, and the \emph{shift}
$\theta^k$ is a measurable mapping of $(\Omega,\calf)$ into itself which is defined for any (possibly random) $k\in\zz$ by
\beq
(\theta^k \omega)_n=\omega_{n+k},\qquad n\in\zz.
\feq
If $E_P(d_1)<\infty$ one can define a ``stationary dual" $\witi \omega$ of the environment $\omega$ as follows \cite{palmain,pdualities}. Without loss of generality, we assume that the underlying probability space supports a random variable $U$, which is independent of $\omega$ and is distributed uniformly on
the interval $[0,1]\subset\mathbb{R}$. For $x\in\rr,$ let $\lfloor x \rfloor$ denote the integer part of $x,$ that is $\lfloor x \rfloor=\sup\{n\in\zz:n\leq x\}.$
We now define $\bigl(\witi a_n, \witi \omega_n)_{n\in\zz}$ by setting
\beq
\witi a_n=a_n+\lfloor U d_0 \rfloor\qquad \mbox{and}\qquad
\witi
\omega_n=
\left\{
\begin{array}{ll}
\lambda_k&\mbox{if}~n=\witi a_k~\mbox{for some}~k\in\zz,\\
1/2&\mbox{otherwise}.
\end{array}
\right.
\feq
Let $\cala=(a_n)_{n\in\zz}$ be the set of marked sites of the integer lattice and
let $\witi \cala=(\witi a_n)_{n\in\zz}$ denote its randomly shifted version introduced above. Furthermore, let
\beq
e_n=\ind{n\in\cala}\qquad\mbox{and}\qquad \witi e_n=\ind{n\in\witi \cala},\quad n\in\zz,
\feq
and let $\Upsilon:=(e_n,\omega_n)_{n\in\zz}$ and $\witi \Upsilon:=\bigl(\witi e_n,\witi \omega_n\bigr)_{n\in\zz}.$
Notice that this construction implies the identity $\bigl(\witi e_n,\witi \omega_n\bigr)_{n\in\zz}=\bigl(\theta^{-M} e_n,\theta^{-M} \omega_n\bigr)_{n\in\zz},$ where $M:=\lfloor U d_0 \rfloor.$
\par
For sparse environments induced by a renewal sequence $a_n,$ the dual environment can be defined equivalently in a rather explicit manner as
a functional of an auxiliary Markov chain. We will exploit this alternative construction in Section~\ref{stable-proofs}.
The uniqueness of the dual environment (which implies, in particular, that the alternative construction yields the same dual) follows
from the reverse ``stationary to cycle-stationary" Palm duality described, for instance, in \cite[Theorem~1]{pdualities}.
\par
The following theorem is an adaptation to our setting of the classical Palm dualities \cite[Chapter~8]{palmain}
between the distribution of $\Upsilon$ under $P$ and the distribution of $\witi \Upsilon$ under a measure $Q$ equivalent to $P$.
\begin{theorem}[see Theorem~2 in \cite{pdualities}]
\label{theorem-a}
Assume that $(\lambda_n,d_n)_{n\in\nn}$ is a stationary ergodic sequence under $P$ and $E_P(d_0)<\infty.$ Define a new probability measure $Q$
on the Borel subsets of the product set $\bigl(\{0,1\}\times (0,1)\bigr)^\zz$ by setting
\beq
\frac{dQ}{dP(\upsilon)}=\frac{d_0(\upsilon)}{E_P(d_0)}, \qquad \upsilon \in \bigl(\{0,1\}\times (0,1)\bigr)^\zz.
\feq
Then:
\item [(a)] $(\witi e_n,\witi \omega_n)_{n\in\zz}$ is  a stationary and ergodic sequence under $Q.$
\item [(b)] $P(\cala\in \cdot\,)=Q\bigl(\witi \cala\in \cdot\,|0 \in\witi\cala\bigr).$
\end{theorem}
We remark that although the claim that the sequence $\witi \Upsilon_n=(\witi e_n,\witi \omega_n)$ is ergodic is not explicitly made in \cite{pdualities}, it can be deduced, for instance, from the result of Exercise~1 in \cite[p.~56]{ergodic-book}. The following is a straightforward corollary to Theorem~\ref{theorem-a}. For the sake of completeness,
the proof is given in the Appendix.
\begin{corollary}
\label{cor-1}
Under the conditions of Theorem \ref{theorem-a}, we have:
\item [(a)] $E_P(d_0^2)=E_P(d_0)\cdot E_Q(d_0).$
\item [(b)] $E_Q(d_0)=2 E_Q(\witi a_0)+1.$
\end{corollary}
We remark that the identity $E_Q(d_0)=2 E_Q(a_0)+1$
can be thought of as a variation of the ``waiting time paradox" of the classical renewal theory \cite{durrett}.

\section{Asymptotic behavior of RWSRE}
\label{main}
In this section, we state the basic limit theorems that describe the asymptotic behavior of the random walk $X_n,$
while the proofs are provided in Section~\ref{proofs}. We first state recurrence and transience criteria for RWSRE.
\subsection{Recurrence and transience criteria}
\label{recs}
Let $\sigma_0=0$ and
\beq
\sigma_n=\inf\{k\in\nn:k>\sigma_{n-1}~\mbox{and}~X_k\in \cala\}.
\feq
Thus $(\sigma_n)_{n\in\mathbb{N}}$ consists of the times of successive visits of $X_n$ to the random point set $\cala.$ Define a nearest-neighbor random walk $(\xx_n)_{n\geq 0}$ on $\zz$ by setting
\beqn
\label{arw}
\xx_n=k\quad\mbox{if and only if}\quad X_{\sigma_n}=a_k.
\feqn
Taking into account the solution of the gambler's ruin problem for the  simple symmetric random walk, we note that $\xx_n$ is a RWRE with quenched transition probabilities given by
\beqn
\label{amat}
P_\omega(\xx_{n+1}=j|\xx_n=i)=
\left\{
\begin{array}{ll}
\xi_i\cdot\fracd{1}{d_i}&\mbox{if}~j=i+1\\
(1-\xi_i)\cdot\fracd{1}{d_{i-1}}&\mbox{if}~j=i-1\\
\xi_i\cdot \fracd{d_i-1}{d_i}+(1-\xi_i)\cdot \fracd{d_{i-1}-1}{d_{i-1}}&\mbox{if}~j=i\\
0&\mbox{otherwise}.
\end{array}
\right.
\feqn
Moreover, $\limsup_{n\to\infty}X_n=\limsup_{n\to\infty} \xx_n$ and $\liminf_{n\to\infty} X_n=\liminf_{n\to\infty} \xx_n,$ $\pp-\as$
Thus, under very mild conditions, recurrence and transience criteria for the RWSRE $X_n$ can be derived directly from the corresponding criteria for the RWRE $\xx_n$ (see, for instance,
\cite[Theorem~2.1.2]{notes} for the latter). More precisely, we have:
\begin{theorem}
\label{rec-thm}
Suppose that the following three conditions are satisfied:
\begin{enumerate}
\item The sequence of pairs $(d_n,\lambda_n)_{n\in\zz}$ is stationary and ergodic
\item $E_P(\log \xi_0)$ exists (possibly infinite)
\item  $E_P\bigl(\log d_0)<\infty.$
\end{enumerate}
Then:
\begin{itemize}
\item [(a)] $E_P(\log\xi_0)<0$ implies $\lim_{n\to\infty} X_n=+\infty,$ $\pp-\as$
\item [(b)] $E_P(\log\xi_0)>0,$ implies $\lim_{n\to\infty} X_n=-\infty,$ $\pp-\as$
\item [(c)] $E_P(\log\xi_0)=0$ implies $\liminf_{n\to\infty} X_n=-\infty$ and $\limsup_{n\to\infty} X_n=+\infty,$ $\pp-\as$
\end{itemize}
\end{theorem}
Theorem \ref{rec-thm} implies that as long as $E_P(\log d_0)$ is finite, the sparse environment $\omega$ induces the same recurrence-transience
behavior as the underlying random environment $\lambda.$ The following theorem shows that the opposite phenomenon occurs when $E_P(\log d_0)=+\infty$. Namely, the properties of $\lambda$ are
essentially irrelevant to the basic asymptotic behavior of $X_n$.
\begin{theorem}
\label{rec1-thm}
Suppose that the following conditions hold:
\begin{enumerate}
\item The sequence of pairs $(d_n,\lambda_n)_{n\in\zz}$ is stationary and ergodic
\item The random variables $d_n$ are i.i.d.
\item $E_P\bigl(|\log \xi_0|\bigr)<+\infty$ while $E_P(\log d_0)=+\infty.$
\end{enumerate}
Then, $\liminf_{n\to\infty} X_n=-\infty$ and $\limsup_{n\to\infty} X_n=+\infty,$ $\pp-\as$
\end{theorem}
 We remark  that
$(d_n)_{n\in\nn}$ is not necessarily independent of $(\lambda_n)_{n\in\zz}$. The proof of Theorem~\ref{rec1-thm} is given in Section~\ref{rec1-proofs}.
\par
%{\bf Maybe adding at the last paragraph: ingredients of the proof, use of the result itself in future, $\log$ determines the threshold.}
\subsection{Transient RWSRE: asymptotic speed}
\label{llns}
We now turn our attention to the law of large numbers for $X_n.$
Whenever it exists, $\lim_{n\to\infty}X_n/n$ is referred to as the \emph{asymptotic speed} of the random walk.
Let $T_0=0$ and for $n\in \nn,$
\beqn
\label{tean}
T_n=\inf\{k\geq 0:X_k=n\}\qquad \mbox{and} \qquad \tau_n=T_{a_n}-T_{a_{n-1}}.
\feqn
Let
\beqn
\label{rols}
\ol S=1+2\sum_{i=0}^{\infty}\prod_{j=0}^i\xi_j\qquad \mbox{and}\qquad \ol F=1+2\sum_{i=1}^{\infty}\prod_{j=0}^{i-1}\xi_{-j}^{-1}.
\feqn
We have the following:
\begin{theorem}
\label{lln-thm}
Let the conditions of Theorem~\ref{rec-thm} hold. Suppose in addition that $(d_n)_{n\in\zz}$ is independent of $(\lambda_n)_{n\in\zz}$ under $P.$
Then the asymptotic speed of the RWSRE exists $\pp-\as$ Moreover,
\beq
\pp\bigl(\lim_{n\to\infty}X_n/n=\mbox{\rm v}_P\bigr)=\pp\bigl(\lim_{n\to\infty}T_n/n=1/\mbox{\rm v}_P\bigr)=1,
\feq
where $\mbox{\rm v}_P\in(-1,1)$ is a constant whose reciprocal $\mbox{\rm v}_P^{-1}$ is equal to
\beqn
\label{aspeed}
\fracd{1}{\mbox{\rm v}_P}&=&\ind{\lim_{n\to\infty} X_n=+\infty}\Bigl[\fracd{\mbox{\rm VAR}_P(d_0)}{E_P(d_0)}+E_P(\ol S)\cdot E_P(d_0)\Bigr]
\\
\nonumber
&&
\quad
-\ind{\lim_{n\to\infty} X_n=-\infty}\cdot\Bigl[\fracd{\mbox{\rm VAR}_P(d_0)}{E_P(d_0)}+E_P(\ol F)\cdot E_P(d_0)\Bigr],\qquad\pp-\as
\feqn
\end{theorem}
\bigskip
\noindent
Notice that if $\lambda_i$ (and hence $\xi_i)$ are i.i.d., then \eqref{aspeed} reduces to
\beq
\fracd{1}{\mbox{\rm v}_P}&=&\ind{\lim_{n\to\infty} X_n=+\infty}\cdot \Bigl[\fracd{\mbox{\rm VAR}_P(d_0)}{E_P(d_0)}+E_P(d_0)\cdot \frac{1+E_P(\xi_0)}{1-E_P(\xi_0)}\Bigr]
\\
\nonumber
&&
\quad
-\ind{\lim_{n\to\infty} X_n=-\infty}\cdot\Bigl[\fracd{\mbox{\rm VAR}_P(d_0)}{E_P(d_0)}+E_P(d_0)\cdot \frac{1+E_P(\xi_0)}{1-E_P(\xi_0)}\Bigr],\qquad\pp-\as
\feq
In order to compare \eqref{aspeed} with the corresponding result for the regular RWRE, note that if $\lim_{n\to\infty} X_n=+\infty,$ $\pp-as,$ then \eqref{aspeed}  reduces to
\beq
\fracd{1}{\mbox{\rm v}_P}=\fracd{\mbox{\rm VAR}_P(d_0)}{E_P(d_0)}+E_P(d_0)\cdot E_P(\ol S),\qquad\pp-\as
\feq
Recall the dual environment $\witi \omega$ defined in Section~\ref{dual}, and let
\beqn
\label{witis}
\witi \rho_n=\frac{1-\witi\omega_n}{\witi\omega_n},\quad n\in\zz, \qquad \mbox{and}\qquad \witi S=\sum_{i=1}^\infty (1+\witi \rho_{-i})\prod_{j=0}^{i-1}\witi \rho_{-j}+1+\witi \rho_0.
\feqn
It is well known that the asymptotic speed of the usual RWRE  is given by $1\bigl\slash E_P\bigl(\witi S\bigr)$ (see, for instance, \cite[Theorem~2.1.9]{notes}).
The proof of the following proposition is straightforward, and it is provided in the Appendix.
\begin{proposition}
\label{compare-speed}
Let the conditions of Theorem~\ref{lln-thm} hold. Suppose in addition that
\begin{enumerate}
\item $\lim_{n\to\infty} X_n=+\infty,$ $\pp-\as$
\item $E_P(d_0)<\infty.$
\end{enumerate}
Then, $\mbox{\rm v}_P=1\bigl\slash E_Q\bigl(\witi S\bigr).$
\end{proposition}
We remark that a proposition similar to Proposition~\ref{compare-speed} can be obtained when the random walk is transient to the left (i.e., when $\lim_{n\to\infty} X_n=-\infty,$ $\pp-\as$) by replacing
$\witi \rho_{-k}$ with $\rho_k^{-1}$ in the formula \eqref{witis}  for $\witi S.$
\par
Theorem~\ref{lln-thm} immediately yields the following version of Theorem~1.3 and Corollary~1.4 in \cite{fixede}.
For any constants $\mu> 0$ and $\nu\geq 0$, we denote by $\calp^\circ_{\mu,\nu}$ the set of distributions $(\lambda_n,d_n)_{n\in\zz}$
for which the conditions of Theorem~\ref{lln-thm} hold and
\beq
E_P(d_0)=\mu \qquad \mbox{and} \qquad 1\bigl\slash E_P(\ol S)=\nu.
\feq
We then have:
\begin{corollary}
\label{pi-lambda}
$\max_{P\in \calp^\circ_{\mu,\nu}} \mbox{\rm v}_P=\nu/\mu.$ Furthermore, the maximum is attained at $P\in \calp^\circ_{\mu,\nu}$ if
and only if $\mbox{\rm VAR}_P(d_0)=0.$
\end{corollary}
Combining this result with \cite[Theorem~4.1]{more}, we obtain the following corollary. For any constants $\mu> 0$ and $b<0$, we denote by $\calp^\ast_{\mu,\nu}$ the set of distributions $(\lambda_n,d_n)_{n\in\zz}$ for which the conditions of Theorem~\ref{lln-thm} hold and
\beq
E_P(d_0)=\mu \qquad \mbox{and} \qquad E_P(\log \xi_0)=b.
\feq
\begin{corollary}
\label{pi-lambdab}
We have:
\beq
\max_{P\in \calp^\ast_{\mu,\nu}} \mbox{\rm v}_P=\frac{1}{\mu}\cdot \frac{1-e^b}{1+e^b}.
\feq
Furthermore, the maximum is attained at $P\in \calp^\circ_{\mu,\nu}$ if
and only if $\mbox{\rm VAR}_P(d_0)=\mbox{\rm VAR}_P(\lambda_0)=0,$ in which case $\lambda_0=\frac{1}{1+e^b},$ $P-\as$
\end{corollary}
The slowdown of a one-dimensional random walk in a random environment, as  compared to a simple random walk, is a well-known general phenomenon \cite{speedups,zsurvey,notes} that can be explained heuristically by fluctuations in the associated random potential. For example, a random walk transient to the right will quickly pass  stretches of the environment that ``push" it forward, but will be ``trapped" for a long time in atypical stretches that ``push" it backward. The situation is different in higher dimensions. See, for instance, \cite{sabot}.
\subsection{Stable limit laws for transient RWSRE}
\label{stable1}

The aim of this section is to derive non-Gaussian limit laws for transient random walks in a sparse random environment.
The existence of the stationary dual environment suggests that the limit theorems can be first obtained for the random walk in the dual environment and then translated into the corresponding results for the RWSRE. In what follows, we adopt this approach even though it has the shortcoming of restricting our derivation to a class of i.i.d. environments for which stable laws in the dual setting are known. It appears plausible that alternative methodologies, which would be considerably more technically involved, such as a direct generalization of the ``branching process" approach of \cite{kks,stable}, or an adaptation of the ``random potential" method developed in \cite{spotential}, would allow to extend the results presented in this chapter to a larger class of i.i.d. environments (and also perhaps to some Markov-dependent environments).
\par
We will adopt here the following set of assumptions:
\begin{assume}
\label{measure1}
\item [(A1)] $(\lambda_n)_{n\in\zz}$ is an i.i.d. sequence
\item [(A2)] $(d_n)_{n\in\zz}$ is an i.i.d. sequence independent of $(\lambda_n)_{n\in\zz}$
\item [(A3)] $P(\epsilon<\lambda_0<1-\epsilon)=1$ for some $\epsilon \in (0,1/2).$
\item [(A4)] For some $\kappa>0,$
\beqn
\label{kappa239t}
E_P(\xi_0^\kappa)=1
\feqn
\item [(A5)] There exists a constant $M>0$ such that $P(d_0<M)=1.$
\item [(A6)] The distribution of $\log\xi_0$ is non-arithmetic, that is $P(\log\xi_0\in\alpha\zz)<1$ for all $\alpha\in\rr.$
\end{assume}
Notice that (A4) implies by Jensen's inequality that $E_P(\log\xi_0)\leq 0.$ In view of (A6), the inequality is strict and hence the random walk is transient to the right.
We remark that although condition (A5) appears to be required for our proof, it is likely that it can be relaxed or even omitted.
\par
For any $\kappa \in (0,2]$ and $b >0$, we denote by $\call_{\kappa,b}$
the stable law of index $\kappa$ with the characteristic function
\beqn
\label{kappa-law}
\log \widehat \call_{\kappa,b}(t)=
-b|t|^\kappa\Bigl(1+i\fracd{t}{|t|} f_\kappa(t)\Bigr),
\feqn
where $f_\kappa(t)=-\tan \fracd{\pi}{2}\kappa$ if $\kappa \neq 1$ and $f_1(t)=2/\pi \log t.$
With a slight abuse of notation we use the same symbol for the
distribution function of this law. If $\kappa<1,$ $\call_{\kappa,b}$ is supported on the positive reals, and if
$\kappa \in (1,2],$ $\call_{\kappa,b}$ has zero mean \cite[Chapter~1]{samor-taqqu}. For $\kappa=2,$ the law $\call_{2,b}$
is a normal distribution with zero mean and variance equal to $2b.$
\par
We have:
\begin{theorem}
\label{stable-thm} Suppose that Assumption \ref{measure1} is satisfied.
Then the following hold for some $b>0:$ \item[(i)] If $\kappa \in (0,1),$
then $\lim_{n \to \infty}\pp\left(n^{-\kappa}X_n \leq \mathfrak{z}
\right)=1-\call_{\kappa,b}(\mathfrak{z}^{-1/\kappa}),$ \item[(ii)]
If $\kappa=1,$ then $\lim_{n \to \infty}\pp\bigl(n^{-1}(\log
n)^2(X_n-\delta(n)) \leq \mathfrak{z}
\bigr)=1-\call_{1,b}(-\mathfrak{z}),$ for suitable $A_1>0$ and
$\delta(n) \sim (A_1\log n)^{-1} n,$ \item[(iii)] If $\kappa \in
(1,2),$ then $\lim_{n \to \infty}
\pp\left(n^{-1/\kappa}\left(X_n-n \mbox{\em v}_P\right) \leq
\mathfrak{z}\right)=1- \call_{\kappa,b}(-\mathfrak{z}),$
\item[(iv)] If $\kappa=2,$ then $\lim_{n \to \infty} \pp\left((n
\log n)^{-1/2} (X_n-n \mbox{\em v}_P)\leq
\mathfrak{z}\right)=\call_{2,b}(\mathfrak{z}).$
\end{theorem}
One can readily see that in the context of Theorem \ref{stable-thm}, if  $\kappa>2,$ then the standard CLT holds (it follows, e.g., from \cite[Theorem 2.2.1]{notes}).
\par
For the hitting times $T_n$, we have:
\begin{proposition}
\label{markov-tau} Let the conditions of Theorem~\ref{stable-thm} hold. Then the
following hold for some $\tilde b>0:$ \item[(i)] If $\kappa \in
(0,1),$ then $\lim_{n \to \infty}\pp\left(n^{-1/\kappa}T_n \leq t
\right)=\call_{\kappa,\tilde b}(t),$ \item[(ii)] If $\kappa=1,$
then $\lim_{n \to \infty}\pp\bigl(n^{-1}(T_n-nD(n)) \leq t
\bigr)=\call_{1,\tilde b}(t),$ for suitable $c_0>0$ and $D(n) \sim
c_0 \log n,$\item[(iii)] If $\kappa \in (1,2),$ then $\lim_{n \to
\infty} \pp\left(n^{-1/\kappa}\left(T_n-n \mbox{\em
v}_P^{-1}\right) \leq t\right)=\call_{\kappa,\tilde b}(t),$
\item[(iv)] If $\kappa=2,$ then $\lim_{n \to \infty} \pp\left((n
\log n)^{-1/2}(T_n-n \mbox{\em v}_P^{-1})\leq
t\right)=\call_{2,\tilde b}(t).$
\end{proposition}
It can be shown that the value of the parameter $b$ in the statement of Theorem~\ref{stable-thm} is solely determined by the distribution of $\lambda_0$; in particular, it is independent of the distribution of $d_0$, provided that $d_0$ satisfies the conditions of the theorem. This result may appear  surprising at first, especially in view of a large deviation interpretation of $\kappa$ given in \cite[Section~2.4]{notes} (it is not hard to see that the rate functions of the random potentials associated with the sequences $\xi_n$ and  $\rho_n$ are actually different). However, it can be explained  in terms of the associated branching process and the corresponding interpretation for $\kappa$.
Furthermore, a careful inspection of the proof given in Section~\ref{stable-proofs} shows that
both parameters $b$ and $\ol b$ of the limiting distributions are
decreasing functions of $E_P(d_0)$ and increasing functions of $VAR(d_n).$ This can be explained by the fact that $b$, in some rigorous sense, plays the role of the variance for the stable laws $\call_{\kappa,b}$; see, for instance, the form of the characteristic function in \eqref{kappa-law} and compare it to the characteristic function of a normal distribution.

\subsection{Limit theorems for recurrent RWSRE}
\label{valleys}

The goal of this section is to obtain a generalization of Sinai's limit theorem for a class of recurrent RWSREs. The main result is stated in Theorem~\ref{sinai-thm}. A suitable normalized random potential for the RWSRE is introduced in Lemma~\ref{potential4}. The notion of a valley of the random potential, which is essential for understanding the behavior of Sinai's model \cite{sinair,notes}, is directly carried over to our setup. The proof of the main result is presented in Section~\ref{sinai-proofs}.
\par
Sinai \cite{sinair} studied a recurrent RWRE $X_n$ and showed that
\beq
\frac{\sigma^2}{(\log n)^2} X_n \Rightarrow b_\infty,
\feq
where $b_\infty$ is a random variable which can be described as the ``location of the deepest valley" of a Brownian motion.
The proof of Sinai \cite{sinair} uses a construction that implements the idea that a properly scaled recurrent RWRE can be thought
of as the motion of a particle in a suitably normalized random potential $W_n.$ The normalized potential converges to a Brownian motion, and Sinai's result shows a remarkable slowing down of the diffusive time scale. The density function of the limit distribution $b_\infty$ was characterized independently by Kesten \cite{Kesten1} and Golosov \cite{localization1,localization}, who obtained that
\beq
P(b_{\infty}\in dx)=\frac{2}{\pi}\sum_{k=0}^{\infty}\frac{(-1)^{k}}{2k+1}\exp\Bigl\{-\frac{(2k+1)^2\pi^2}{8}|x|\Bigr\}dx
\feq

In this paper, we derive a limit theorem for a recurrent random walk in a sparse random environment under the following assumption:
Let $\alpha\in(0,1)$ and assume that $d_{1}$ is in the domain of attraction of a
stable law with index $\alpha.$ Namely,
\beqn
\label{E:assum}
P(d_{1}>t)=t^{-\alpha}h(t), \qquad t\geq 1,
\feqn
where $h(t)$ is slowly varying at infinity, that is $h(\lambda t)\sim h(t)$ as $t$ goes to $+\infty$ for all $\lambda >0.$ In particular, we
define $S_{n}=n^{-1}\sum_{k=1}^{n-1}\log\xi_k$ and assume the following:
\begin{assume}
\label{assume}
\item[1.] $E_P(\log\xi_0)=0$ (recurrence)
\item[2.] $\sigma_P^2:=E_P(\log^2\xi_0)\in (0,\infty)$
\item[3.] $\frac{1}{\sqrt{n}}\sum_{k=1}^{[nt]}\log\rho_{k}\Rightarrow B(t)$
\item[4.]  $P(d_{1}>t)= t^{-\alpha}h(t)$, where $\alpha\in (0,1)$ and $h(t)$ is slowly varying.
\end{assume}
Recall that a function $f:\rr_+\to\rr$ is said to be {\em regularly varying of index} $\alpha\in\rr$
if $f(t)=t^\alpha h(t)$ for a slowly varying $h:\rr_+\to\rr.$ We denote the set of all regularly varying functions of index $\alpha$ by $\calr_\alpha.$
We have the following:
\begin{theorem}
\label{sinai-thm}
Let Assumption~\ref{assume} hold and fix any $\delta>0.$ Then, there is a function $u\in\calr_{2/\alpha}$ such that that the following holds:
For any $\veps>0$ and $\delta\in(0,1),$ there is an integer $n_1$ such that for all $n>n_1$ there exist
a set of environments $C_n\subset \Omega$ and a random variable $b_n=b_n(\omega)$ such that $P(C_n)\geq 1-\delta$
and
\beq
\lim_{n\to\infty}P_\omega\Bigl(\Bigl|\frac{X_n}{u(\log n)}-b_n\Bigl|>\veps\Bigl)=0
\feq
uniformly in $\omega\in C_n$. Moreover, as $n\to\infty$ the probability distribution for $b_n$ converges weakly to a non-degenerate limiting distribution $b_\infty.$
\end{theorem}
We remark that Sineva \cite{alpha,preds} obtained similar limit laws for different variations of Sinai's model. In all these results, the limiting distribution of a properly scaled random walk $X_n$ admits a representation as the deepest valley of an auxiliary process, which in turn is obtained as the weak limit of a suitably defined random potential. For a definition of the notion of a valley in this context, we refer the reader to \cite{sinair} or  \cite{notes}.

\subsection{Environment viewed from the position of the particle}
\label{view}

In this section we study the ``environment viewed from the
particle'' process $(\theta^{X_n} \omega)_{n \geq 0}$ for a transient RWSRE. It is not hard to see
that the pair $(\theta^{X_n} \omega,X_n)$ forms a Markov chain, which allows
to consider $X_n$ as a functional (projection into the second coordinate) of a Markov process.
Even though the state space of this Markov chain is considerably large, the representation is
useful due to the fact that the underlying Markov chain turns out to be stationary and ergodic in the transient regime.
The concept of the environment viewed from the particle was introduced by S.~Kozlov in a broader context in \cite{kozlov} (see also
\cite{zsurvey} and \cite{bremont,bremont-l,szlln}). In Section~\ref{llns}, we proved the existence of the asymptotic speed $\mbox{v}_P:=\lim_{n\to\infty} X_n$ for RWSREs associated with a stationary and ergodic environment $(d_n,\lambda_n)_{n\in\zz}$ by using a direct approach. In fact, using the techniques described in \cite[Section~2.2]{notes} and the existence of the dual environment, one can prove the following result. Similarly to \eqref{rhoxi}, let
\beqn
\label{xin}
\xi_n=\frac{1-\lambda_n}{\lambda_n},\qquad n\in\zz.
\feqn
We have:
\begin{theorem}
\label{character}
Consider a random walk $X_n$ in a stationary and ergodic sparse environment $(\lambda_n,d_n)_{n\in\zz}.$
Assume that $E_P(\log \xi_0)$ is well defined (possibly infinite) and $E_P(d_0)<\infty.$
Then
\begin{itemize}
\item [(a)]
$\mbox{\em v}_{_P}>0$ if and only if there exists a stationary distribution $P^\circ$ equivalent
to $P$ for the Markov chain $\ol \omega_n=\theta^{X_n} \omega,$ $n \geq 0.$ If such a distribution $P^\circ$ exists it is unique and is given by the following formula:
\beqn
\label{qm}
P^\circ(B)=\mbox{\rm v}_{_P} E_Q\Bigl[E_\omega^0\Bigl(\sum_{n=0}^{T_1-1}\ind{\ol \omega_n \in B}\Bigr)\Bigr],~~~~~B \in \calf,
\feqn
where $\ol \omega_n:=\theta^{\xi_n}\omega$ and $Q$ is the distribution of the dual environment.
\item [(b)] $(\ol \omega_n)_{n\geq 0}$ is an ergodic process under $\pp^\circ:=P^\circ\otimes P_\omega.$
\item [(c)] $\frac{dP^\circ}{dP}=\frac{dQ}{dP}\times \Lambda(\omega)=\frac{d_0\cdot \Lambda(\omega)}{E_P(d_0)},$ where
\beqn
\label{lambda}
\Lambda(\omega):=\frac{1}{\omega_0}\Bigl[1+\sum_{i=1}^\infty \prod_{j=1}^i \rho_j\Bigr]
=\frac{1}{\omega_0}\Bigl[d_1+\sum_{i=1}^\infty d_{i+1}\times \prod_{j=1}^i \xi_j\Bigr].
\feqn
\item [(d)] ${\mbox{\rm v}}_P=1\slash E_Q(\Lambda)=\frac{E_P(d_0)}{E_P(d_0\Lambda)}.$
\end{itemize}
\end{theorem}
With one exception, the proof of Theorem \ref{character} follows along the lines of the corresponding results in
\cite[Section~2.1]{notes} (namely, Lemmas~2.1.18, 2.1.20, 2.1.25, and Corollary~2.1.25 therein). The only exception
is the proof that the existence of the environment viewed from the position of the particle actually implies
$\mbox{v}_P>0.$ The latter can be obtained by a straightforward modification of the proofs of
\cite[Theorem 3.5 (ii)]{bremont} or \cite[Theorem~2.3]{strip} for instance. The proof of Theorem~\ref{character} is therefore omitted.
\begin{remark}
The asymptotic speed for the simple nearest-neighborhood random walk on $\zz$ with probability of jumps forward $p$
and jumps backward $q=1-p$ is $(p-q)=2p-1.$ Although $\mbox{\rm v}_P$ is not equal to $E_Q(2\omega_0-1),$ quite remarkably it turns out to be equivalent to $E_{P^\circ}(2\omega_0-1)$ (compare, for instance, with formula (2.1.29) in \cite{notes}).
\end{remark}

\section{Proofs}
\label{proofs}
\subsection{Proof of Theorem~\ref{rec1-thm}}
\label{rec1-proofs}
Recall the definitions of $\rho_n$ and $\xi_n$ in \eqref{rhoxi}. Furthermore, for a non-zero integer $n$ let $\eta_n$ be the number of marked sites within the closed interval $I_n=\mbox{sign}(n)\cdot [1,n].$ More precisely, let
$\eta_0=0$ and
\beqn
\label{etan}
\eta_n=\chi(I_n \cap \cala)=
\left\{
\begin{array}{lc}
\sum_{k=1}^n \ind{k\in \cala}&\mbox{if}~n>0\\
$\mbox{}$&\\
\sum_{k=1}^n \ind{-k\in \cala}&\mbox{if}~n<0.
\end{array}
\right.
\feqn
Notice that by the ergodic theorem,
\beqn
\label{renewal}
\lim_{|n|\to\infty} \frac{\eta_n}{|n|}=Q(0\in\cala)=\frac{1}{E_P(d_1)},\qquad P-\as~\mbox{and}~Q-\as
\feqn
Denote
\beq
S(\omega)=\sum_{k=1}^\infty\rho_1\rho_2\cdots \rho_k\qquad\mbox{and}\qquad F(\omega)=\sum_{k=0}^\infty\rho_0^{-1}\rho_{-1}^{-1}\cdots \rho_{-k}^{-1}.
\feq
To prove Theorem~\ref{rec1-thm} it suffices (see, for instance, the proof of \cite[Theorem~2.1.2]{notes}) to show that the conditions of the theorem imply
\beqn
\label{exitp}
P\bigl(S(\omega)=F(\omega)=+\infty\bigr)=1
\feqn
\begin{remark}
The functions $S(\omega)$ and $F(\omega)$ appear in the solution of the gambler's ruin problem for an infinite box. Therefore, they
are related to the basic recurrence-transience properties of the random walk (see, e.g.,  \cite[Theorem~2.1.2]{notes}). In particular, \eqref{exitp} implies recurrence.
\end{remark}
Toward this end, note that $\eta_{a_n}=\eta_{a_n+1}=\ldots=\eta_{a_{n+1}-1}=n$ for $n\geq 0,$ and hence
\beqn
\label{curious}
S(\omega)=\sum_{k=1}^{\infty}\xi_1\xi_2\cdots \xi_{\eta_k}=(a_1-1)+\sum_{n=1}^{\infty}\xi_1\xi_2\cdots \xi_n \cdot d_{n+1},
\feqn
where, to claim the first identity, we used the standard convention that $\xi_1\xi_2\cdots \xi_{\eta_k}=1$ if $\eta_k=0.$ Similarly,
$\eta_{a_n}=\eta_{a_n+1}=\ldots=\eta_{a_{n+1}-1}=n+1$ for $n<0,$ and hence
\beqn
\label{curious1}
F(\omega)=\sum_{k=1}^{\infty}\xi_0^{-1}\xi_{-1}^{-1}\xi_{-2}^{-1}\cdots \xi_{-\eta_k}^{-1}
=\sum_{n=0}^{\infty}\xi_0^{-1}\xi_{-1}^{-1}\cdots \xi_{-n}^{-1} \cdot d_{-n}.
\feqn
Furthermore, the condition $E_P(\log d_0)=+\infty$ implies that $\sum_{n=1}^\infty P(\log d_0>M\cdot n)=\infty$ for any $M>0.$
Thus, since $d_n$ are i.i.d., it follows from the second Borel-Cantelli lemma that $P(\log d_n>M\cdot n~\io)=1$ for any $M>0.$ Hence, the ergodic theorem
along with the condition $E_P\bigl(|\log \xi_0|\bigr)<+\infty$ imply that, $P-\as,$
\beq
\limsup_{n\to \infty}\frac{1}{n}\log \Bigl(d_{n+1}\cdot \prod_{k=1}^n \xi_k\Bigr)=
\limsup_{n\to\infty}\frac{1}{n}\Bigl(\sum_{k=1}^n\log\xi_k+\log d_{n+1}\Bigr)=+\infty,
\feq
which yields $P\bigl(S(\omega)=+\infty\bigr)=1.$ A similar argument shows that, under the conditions of the theorem, $P\bigl(F(\omega)=+\infty\bigr)=1$ and hence \eqref{exitp} holds, as desired.
\qed
\subsection{Proof of Theorem~\ref{lln-thm}}
\label{lln-proofs}
The proof is an adaption of the corresponding arguments for the regular RWRE. See, for instance, \cite[Section 2.2]{notes}.
\par
Recall the definitions of $T_n$ and $\tau_n$ in \eqref{etan}. We have:
\begin{lemma}
\label{ste-lemma}
Assume that the conditions of Theorem~\ref{rec-thm} hold and suppose, in addition, that $\,\pp(\limsup_{n\to\infty} X_n=1).$
Then $(\tau_n)_{n\in\nn}$ is a stationary and ergodic sequence under the law $\pp.$
\end{lemma}
\begin{proof}[Proof of Lemma~\ref{ste-lemma}]
Enlarge, if needed, the underlying probability space to include a sequence of i.i.d. random variables $\gamma=(\gamma_{x,n})_{x\in\zz,n\in\nn}$ such that
\begin{enumerate}
\item $\gamma$ is independent of $\omega$ under the law $P,$ \quad and
\item each random variable $\gamma_{x,n}$ is distributed uniformly on the interval $[0,1].$
\end{enumerate}
Let $\nn_0$ denote the set of non-negative integers, that is $\nn_0=\nn\cup \{0\}.$
For $x\in\zz$ and $n\in\nn_0,$ let $l_x(n)=\sum_{t=0}^n \ind{X_t=x}$ be the number of visits of the random walk to the site $x$ by the time $n.$
For $n\in\nn_0,$ denote $l_n=l_{X_n}(n),$ $\gamma_n=\gamma_{X_n,l_n},$ and $\ol \omega_n=\omega_{X_n}.$
Without loss of generality, we can assume that $X_n$ is defined recursively as follows:
\beq
X_{n+1}=X_n+\ind{\gamma_n<\ol \omega_n}-\ind{\gamma_n \geq \ol \omega_n}.
\feq
For $n\in\nn,$ let $C_n=(\xi_x,d_j,\Gamma_x)_{j,x\leq n},$ where $\Gamma_x=(\gamma_{k,i})_{k\leq x,i\in\nn_0}.$
The sequence $(\tau_n)_{n\in\nn}$ defined by \eqref{tean} is stationary under $\pp$ because $(\lambda_n,d_n)_{n\in\zz}$ is stationary and $\pp(T_n<\infty)=1$
for all $n>0$ under the conditions of the lemma. Furthermore, in the enlarged probability space $(\tau_n)_{n\in\nn}$ becomes a deterministic function of $(C_k)_{k\leq n}.$ This completes the proof of the lemma since the sequence $(C_n)_{n\in\nn}$ is stationary and ergodic under $P$ in the enlarged probability space.
\end{proof}
Under the conditions of Lemma~\ref{ste-lemma}, the ergodic theorem yields
\beq
\frac{T_{a_{n}}}{n}=\frac{1}{n}\sum_{i=1}^n\tau_{a_{i}} \to \mathbb{E}(\tau_{a_{1}})\quad \mbox{as}~n\to\infty, \qquad \pp-\as
\feq
We have
\begin{lemma}
\label{average-lemma}
Assume that the conditions of Theorem~\ref{rec-thm} hold and suppose, in addition, that $\lambda$ and $\cala$ are independent under $P.$ Then:
\begin{itemize}
\item [(a)] $\ee(T_{a_1})=\mbox{\rm VAR}_P(d_1)+E_P(\ol S)\cdot [E_P(d_1)]^2,$
\item [(b)] $\ee(T_{a_{-1}})=\mbox{\rm VAR}_P(d_1)+E_P(\ol F)\cdot [E_P(d_1)]^2.$
\end{itemize}
\end{lemma}
\begin{proof}[Proof of Lemma~\ref{average-lemma}]
We will only prove the result in (a), the proof of (b) being similar.
To evaluate $T_{a_1},$ we will use a decomposition of the paths of the random walk according to its first step:
\beqn
\label{tean-decomp}
&&
T_{a_1}=1+\ind{X_1=1}[\ind{\witi T_0<\witi T_{a_1}}(\witi T_0+T_{a_1}')
+\ind{\witi T_0>\witi T_{a_1}}\witi T_{a_1}]
\\
\nonumber
&&
\quad
+
\ind{X_1=-1}[\ind{\widehat T_0<\widehat T_{a_{-1}}}(\widehat T_0+T_{a_1}'')
+\ind{\widehat T_0>\widehat T_{a_{-1}}}(\widehat T_{a_{-1}}+T_0'+T_{a_1}''')],
\feqn
where
\beq
\begin{array}{ll}
\witi T_0=\inf\{n>T_1:X_n=0\},
&
\witi T_0+T_{a_1}'=\inf\{n>\witi T_0: X_n=a_1\},
\\
\witi T_{a_1}=\inf\{n>T_1: X_n=a_1\},
&
\widehat T_0=\inf\{n>T_{-1}:X_n=0\},
\\
\widehat T_0+T_{a_1}''=\inf\{n>\widehat T_0: X_n=a_1\},
&
\widehat T_{a_{-1}}=\inf\{n>T_{-1}: X_n=a_{-1}\},
\\
\widehat T_{a_{-1}}+T_0'=\inf\{n>\widehat T_{a_{-1}}: X_n=0\},
&
\widehat T_{a_{-1}}+T_0'+T_{a_1}'''=\inf\{n>T_{a_{-1}}+T_0':X_n=a_1\}.
\end{array}
\feq
Taking quenched expectations $E_\omega(\cdot)$ in both sides of \eqref{tean} yields
\beq
E_\omega(T_{a_1})=1&+&\lambda_0[\cale_1(T_0 \wedge T_{a_1})+\calp_1(T_0<T_{a_1})E_\omega(T_{a_1})]
\\
&+&(1-\lambda_0)[\cale_{-1}(T_0 \wedge T_{a_{-1}})+
E_\omega(T_{a_1})+\calp_{-1}(T_{a_{-1}}<T_0) E_{\omega,a_{-1}}(T_0)].
\feq
Using the solution of the gambler's ruin problem under $\calp,$ we obtain
\beq
\frac{\lambda_0}{a_1}E_\omega(T_{a_1})&=&1+\lambda_0(a_1-1)+(1-\lambda_0)(|a_{-1}|-1)+\frac{1-\lambda_0}{|a_{-1}|}E_{\omega,a_{-1}}(T_0)
\\
&=&
\lambda_0 a_1 +(1-\lambda_0)|a_{-1}| +\frac{1-\lambda_0}{|a_{-1}|} E_{\omega,a_{-1}}(T_0).
\feq
Thus,
\beq
\frac{1}{a_1}E_\omega(T_{a_1})=a_1 +\xi_0|a_{-1}| +\xi_0\cdot \frac{1}{|a_{-1}|} E_{\omega,a_{-1}}(T_0).
\feq
Iterating yields
\beq
\frac{1}{a_1}E_\omega(T_{a_1})=a_1 +2\sum_{k=0}^\infty \xi_0\xi_{-1}\cdots\xi_{-k}\cdot d_{-k}.
\feq
Taking expectations with respect to $P$, and using a truncation argument similar to that given in the proof of \cite[Lemma~2.1.12]{notes} in order
to verify when $\ee(T_{a_1})<\infty,$ we obtain
\beq
\ee(T_{a_1})=\mbox{VAR}_P(d_1)+\bigl[E_P(d_1)\bigr]^2\cdot E_P(\ol S),
\feq
as desired. This completes the proof of (a) of the lemma. Part (b) can be derived along the same lines, and hence its proof is omitted.
\end{proof}
In view of Lemma~\ref{average-lemma}, we are now in a position to finish the proof of Theorem~\ref{lln-thm}. Variations of Lemma \ref{teax-lemma} below have appeared in a  number of references in the field of random walks in random environments. Nonetheless, we provide a proof for the reader's convenience.
\begin{lemma}
\label{teax-lemma}
Assume that the conditions of Theorem~\ref{rec-thm} hold and suppose, in addition, that $\lim_{n\to\infty}\frac{T_{a_n}}{n}=\alpha,$ $\pp-\as,$
for some constant $\alpha\leq \infty.$ Then,
\beq
\lim_{n\to\infty}\frac{T_n}{n}=\frac{\alpha}{E_P(d_1)}\qquad\mbox{\rm and}\qquad
\lim_{n\to\infty}\frac{X_n}{n}=\frac{E_P(d_1)}{\alpha},  \qquad \pp-\as
\feq
\end{lemma}
\begin{proof}[Proof of Lemma~\ref{teax-lemma}]
First, observe that \eqref{etan} implies
\beq
a_{\eta_n}\leq n <a_{\eta_n+1},\qquad \pp-\as
\feq
Thus, in view of \eqref{renewal},
\beq
\lim_{n\to\infty}\frac{T_n}{n}=\lim_{n\to\infty}\frac{T_{a_{\eta_n}+1}}{n}=
\lim_{n\to\infty}\frac{T_{a_{\eta_n}}}{\eta_n}\cdot \frac{\eta_n}{n}=\frac{\alpha}{E_P(d_1)},  \qquad \pp-\as
\feq
Let now $\zeta(n) \in\zz$ be the unique nonnegative random number such that
\beqn
\label{zeta}
T_{a_{\zeta(n)}}\leq n<T_{a_{\zeta(n)+1}}.
\feqn
Since $X_n$ is transient to the right, $\pp\bigl(\lim_{n\to\infty}\zeta(n)=\infty\bigr)=1.$
Furthermore, \eqref{zeta} implies that
\beq
X_n<a_{\zeta(n)+1}\qquad \mbox{and}\qquad X_n\geq a_{\zeta(n)} -(n-T_{a_{\zeta(n)}}).
\feq
Thus,
\beq
\frac{a_{\zeta(n)}}{n}-\Bigl(1-\frac{T_{a_{\zeta(n)}}}{n}\Bigr)\leq\frac{X_n}{n}<\frac{a_{\zeta(n)+1}}{n}.
\feq
But \eqref{zeta} along with the existence of $\lim_{n\to\infty}\frac{n}{T_n}$ yield
\beq
\lim_{n\to\infty}\frac{a_{\zeta(n)}}{n}=\lim_{n\to\infty}\frac{a_{\zeta(n)}}{T_{a_{\zeta(n)}}}=\lim_{n\to\infty}\frac{n}{T_n}=\frac{E_P(d_1)}{\alpha}, \qquad \pp-\as
\feq
Hence,
\beq
\frac{E_P(d_1)}{\alpha}\leq\liminf_{n\to\infty}\frac{X_n}{n}\leq\limsup_{n\to\infty}\frac{X_n}{n}\leq\frac{E_P(d_1)}{\alpha},
\feq
which implies the result in Lemma~\ref{teax-lemma}.
\end{proof}
The proof of Theorem~\ref{lln-thm} is complete.
\subsection{Proof of Theorem~\ref{stable-thm}}
\label{stable-proofs}
The proof uses the dual Markovian environment and the reduction to stable limit laws for random walks in a Markovian environment obtained
in \cite{stable}. Recall the definition of $T_n$ in \eqref{tean}, and observe that the distribution of the trajectory $(X_n)_{n\in\nn}$ under the law $P$ coincides with the
distribution of $(X_{T_{a_0}+n}-a_0)_{n\in\nn}$ under $Q.$ Moreover,
\beq
\bigl|X_n-(X_{T_{a_0}+n}-a_0)\bigr|\leq a_0+ \bigl|X_n-X_{n+T_{a_0}}\bigr|\leq a_0+T_{a_0}.
\feq
Since the random walk is transient to the right, $Q(T_{a_0}<\infty)=P(T_{a_0}<\infty)=1.$ Therefore,
$\frac{1}{c_n}\cdot \bigl|X_n-(X_{T_{a_0}+n}-a_0)\bigr|$ converges to zero in distribution (under the law $Q$) for any sequence of scaling factors
$(c_n)_{n\in\nn}$ such that $\lim_{n\to\infty} c_n=\infty.$ Thus it suffices to prove the stable limit laws under $Q.$
\par
Towards this end, let
\beq
Y_n=n-a_{\eta_n}=n-\sup\{k\in\zz: k\leq n~\mbox{and}~k\in\cala\},\qquad n\in\zz.
\feq
Notice that $Y_{a_k}=0,$ $k\in\nn,$ and
\beq
Y_{n+1}-Y_n=1 \quad \mbox{if} \quad a_{\eta_n}\leq n <a_{\eta_n+1}.
\feq
Let $\zz_+$ denote the set of non-negative integers, that is $\zz_+=\nn\cup\{0\}.$ If Assumption~\ref{measure1} holds, then the sequence $Y=(Y_n)_{n\in\zz}$  is a positive-recurrent Markov chain on $\zz_+$ under the law $Q.$ Furthermore, the transition kernel $H(x,y)=Q(Y_{n+1}=y|Y_n=x)$  is given by
\beq
H(x,y)=
\left\{
\begin{array}{cll}
\frac{P(d_0>x+1)}{P(d_0>x)}&\mbox{\rm if}& y=x+1,~x\in\zz_+\\
\\
\frac{P(d_0=x+1)}{P(d_0>x)}&\mbox{\rm if}& y=0,~x\in\zz_+\\
\\
0&&\mbox{\rm otherwise.}
\end{array}
\right.
\feq
We briefly remark here that $Y$ is non-homogeneous under the law $P.$ It follows from Theorem~\ref{theorem-a} that $Y$  is a stationary Markov chain on $\zz_+$ under the law $Q$, and the  initial distribution of $Y$ is the (unique) invariant distribution of
$H.$ That is, using the notation of Theorem~\ref{theorem-a},
\beq
Q(Y_0=x)=P\bigl(\lfloor Ud_0 \rfloor=x \bigr)=\frac{P(d_0>x)}{E_P(d_0)},\qquad x\in\zz_+.
\feq
It then follows that under $Q,$ the sequence $(Y,\omega)$ constitutes a two-component Markov chain with transition kernel
depending only on the current value of the first component (but not of the second). More precisely, with probability one,
\beq
Q(Y_{n+1}=y,\omega_{n+1}\in A|Y_n=x,\omega_n=u)=\hh(x;y,A),
\feq
where the stochastic kernel $\hh$ on $\zz_+\times (\zz_+\times \Omega)$ is given by
\beq
\hh(x;y,A)=H(x,y) \cdot \bigl(\ind{\frac{1}{2}\in A}\cdot \ind{y=0}+P(\lambda_0\in A)\cdot \ind{y>0}\bigr),\qquad A\in \calb\bigl([0,1]\bigr).
\feq
Clearly, the reverse chain $(Y_n)_{n\in\zz}$ is an irreducible Markov chain
in the finite state space $\{0,1,\ldots,M-1\}.$  In view of the main results of \cite{stable}, in order to establish that the claims of Theorem~\ref{stable-thm} and Proposition~\ref{markov-tau} hold for the RWSRE $X_n$ and the hitting times $T_n,$ it suffices to verify the following set of conditions for the stationary Markov chain $Y_n$ and the associated Markov hidden model
$(Y_{-n},\lambda_{-n}, d_{-n})_{n\in\zz}:$
\begin{itemize}
\item [(B1)] $\limsup_{n \to \infty}
\fracd{1}{n} \log E_{Q} \left(\prod_{i=0}^{n-1} \rho_{-i}^{\beta_1}
\right) \geq 0 $ and $ \limsup_{n \to \infty}
\fracd{1}{n} \log E_{Q} \left(\prod_{i=0}^{n-1} \rho_{-i}^{\beta_2}
\right)<0 $ for some constants $\beta_1>0$ and $\beta_2>0.$
\item [(B2)] The process $q_n=\log \rho_{-n}$ is non-arithmetic
relative to $(x_n)$ in the following sense: there do not exist a constant $\alpha >0$
and a measurable function $\gamma :\cals \to [0,\alpha)$ such that
\beq 
Q\bigl(q_0 \in \gamma(x_{-1})-\gamma(x_0)+\alpha \cdot \zz\bigr)=1. 
\feq
\end{itemize}
To this end, observe that condition (B1) holds (compare with \cite{stable}) because the following holds true: By virtue of Theorem~\ref{theorem-a}
\beq
&&
\limsup_{n \to \infty}
\fracd{1}{n} \log E_Q \left(\prod_{i=0}^{n-1} \rho_{-i}^{\beta}\right)=\limsup_{n \to \infty}
\fracd{1}{n} \log E_Q \left(\prod_{i=1}^n \rho_i^{\beta}\right)
\\
&&
\qquad
=
\limsup_{n \to \infty}
\fracd{1}{n} \log E_P \left(\prod_{i=1}^n \rho_i^{\beta}\right),
\feq
the function $\Lambda(\beta):=\limsup_{n \to \infty}
\fracd{1}{n} \log E \left(\prod_{i=0}^{n-1} \rho_{-i}^{\beta}\right)$ is convex, $\Lambda(0)=0,$ $\Lambda'(0)=E_p(\log \rho_0)<0,$ the distribution
of $\rho_0$ is non-degenerate, and
\beq
E_P\Bigl(\prod_{i=1}^n \rho_i^\kappa\Bigr)&=&
E_P\Bigl(\prod_{i=0}^{\eta_n} \xi_i^\kappa\Bigr)=E_P\Bigl(E_P\Bigl(\prod_{i=0}^{\eta_n} \xi_i^\kappa|\eta_n\Bigr)\Bigr)
\\
&=&E_P\Bigl(\bigl(E_P(\xi_0^\kappa)\bigr)^{\eta_n}\Bigr)=E_P(1^{\eta_n})=1.
\feq
Furthermore, (A6) of Assumption~\ref{measure1} along with the fact that $\xi_n$ are i.i.d. trivially implies (B2). Notice that the measure $Q$ is absolutely continuous with respect to $P,$ and therefore $P\bigl(q_0 \in \gamma(x_{-1})-\gamma(x_0)+\alpha \cdot \zz
\bigr)=0$ guarantees $Q\bigl(q_0 \in \gamma(x_{-1})-\gamma(x_0)+\alpha \cdot \zz
\bigr)=0.$
The proof is complete.
\subsection{Proof of Theorem~\ref{sinai-thm}}
\label{sinai-proofs}
Let $D(\rr)$ denote the set of real-valued c\`{a}dl\`{a}g functions
on $[0,1]$ equipped with the Skorokhod $J_1$-topology. We use the notation $\Rightarrow$ to denote the weak convergence in $D(\rr).$
We have:
\beq
U_n(t):=\frac{1}{r_n}\sum_{k=1}^{\lfloor nt \rfloor} d_k ~~\mbox{\rm converges}~~\mbox{\rm weakly}~~\mbox{\rm to}~~G_{\alpha}(t),
\feq
where $G_{\alpha}(t)$ is a totally asymmetric stable process with
\beq
E\bigl(e^{i\theta G_{\alpha}(1)}\bigr)=\exp\Bigl\{-|\theta|^{\alpha}\Bigl(1-i~\mbox{\rm sign}(\theta)\tan\Bigl(\frac{\pi\alpha}{2}\Bigr)\Bigr)\Bigr\}, \qquad \theta\in\rr.
\feq
The key ingredient of the proof is the following (functional) limit theorem for a suitably defined random potential.
\begin{lemma}
\label{potential4}
\item [(a)] Assume that condition \eqref{E:assum} holds with $\alpha\in(0,1).$ Then, as $n\to\infty,$
\beq
\frac{1}{\log n}\sum_{k=1}^{\lfloor u(\log n)t\rfloor}\log \rho_k \Rightarrow V_\alpha,
\feq
for some sequence $u(n)\in\calr_{2\alpha}.$
\item [(b)] If $E(d_1)<\infty,$ then
\beq
\frac{1}{\sqrt{n}}\sum_{k=1}^{\lfloor nt \rfloor}\log \rho_k \Rightarrow \mu\sigma_P W,
\feq
where $W$ is a standard Brownian motion and $V_\alpha=\sigma_P W(G_{\alpha}^{-1}).$
\end{lemma}
\begin{proof}[Proof of Lemma~\ref{potential}]
$\mbox{}$
\\
{\bf (a)}
Let
\beq
U_{n}(t):=\frac{1}{r_n}\sum_{k=1}^{\lfloor nt \rfloor} d_k=\frac{1}{r_n}a_{ \lfloor nt \rfloor}
\qquad \mbox{and} \qquad R_n(t):=\frac{1}{\sqrt{n}}\sum_{k=1}^{\lfloor nt \rfloor}\log\xi_k.
\feq
It follows from the assumptions of the lemma that, as $n\to\infty,$
\beq
U_{n}(t)\Rightarrow G_\alpha \qquad \mbox{and} \qquad R_n(t)\Rightarrow \sigma_P W.
\feq
for some sequence $r_n\in\calr_{\alpha+1}.$
Let $U_n^{-1}=n^{-1}\cdot \eta( \lfloor t r_n \rfloor)$ and $G_{\alpha}^{-1}$ be the inverses in $D(\rr_+,\rr)$ of $U_n$ and $G_\alpha,$ respectively.
Then the convergence of $U_n$ and $R_n,$ along with their independence of each other, imply (see, for instance, the derivation of the formula (2.29) in \cite{directrw}) that in $D(\rr_+,\rr),$
\beq
\Bigl(R_n(t),\frac{1}{n}\eta(\lfloor t r_n \rfloor)\Bigr) \Rightarrow \bigl(\sigma_P W,G_{\alpha}^{-1}\bigr),\qquad \mbox{\rm as}~n\to\infty.
\feq
Since the paths of the Brownian motion are continuous, it follows from a random change lemma in \cite[p.~151]{Billingsley}
that $R_{n}\bigl(\frac{1}{n}\eta(\lfloor t r_n \rfloor)\bigr)\Rightarrow \sigma_P W(G_{\alpha}^{-1}(t))$ in $D(\rr_+,\rr).$ That is,
\beq
\frac{1}{\sqrt{n}}\sum_{k=1}^{\eta([r_nt])}\log\xi_k\Rightarrow \sigma_P W(G_{\alpha}^{-1}),\qquad \mbox{\rm as}~n\to\infty,
\feq
and, passing to the subsequence $n_k=\log^2 k,$ $k\in\nn,$ we obtain
\beq
\frac{1}{\log k}\sum_{i=1}^{\eta([r_{\log^2 k}t])}\log\xi_i=\frac{1}{\log k}\sum_{i=1}^{[r_{\log^2 k}t]}\log\rho_i\Rightarrow \sigma_P W(G_{\alpha}^{-1}),\qquad \mbox{\rm as}~k\to\infty.
\feq
To conclude the proof of part (a), notice that $r_n\in\calr_\alpha$ implies that $r_{\log^2 n}=u(\log n)$ for some sequence $u(n)\in\calr_{2\alpha},$ as desired.
\\
$\mbox{}$
\\
{\bf (b)} We now turn to the proof of part (b) of the theorem. Since $E_{p}(d_{1})<\infty,$ the renewal theorem implies
\beq
\frac{\eta_{n}}{n}\to C=\frac{1}{E_{p}(d_{1})}.
\feq
Therefore (see, for instance, Theorem~14.4 in \cite[p.~152]{Billingsley}),
\beq
\frac{1}{\sqrt{n}}\sum_{k=1}^{[nt]}\log \rho_k=\frac{1}{\sqrt{n}}\sum_{k=1}^{\eta([nt])}\log \xi_k=\frac{\sqrt{\eta([nt])}}{\sqrt{n}}\frac{1}{\sqrt{\eta([nt])}}
\sum_{k=1}^{\eta([nt])}\log \xi_k \Rightarrow \mu\sigma_P B(\,\cdot\,),
\feq
where the convergence is the weak convergence in the Skorohod space $D(\rr_+).$
\end{proof}
Define the normalized random potential associated with the sparse environment as the following:
\beqn
\label{ar-hat}
\widehat R_n(t)=\sign(t)\cdot \frac{1}{\log n}\sum_{k=1}^{[u(\log n)t]}\log \rho_k=\sign(t)\cdot\frac{1}{\log n}\sum_{k=1}^{\eta{[u(\log n)t]}}\log \xi_k.
\feqn
By Lemma~\ref{potential}, $\{\widehat R_n(t):t\geq 0\}$ converges weakly in $D(\rr_+,R)$ to the process $V_\alpha.$
One can now proceed as in \cite{alpha} in order to establish Theorem~\ref{sinai-thm}. In fact, the proof of the main result
in \cite{alpha} is an adaptation of the original argument of Sinai \cite{sinair} to a situation where the random  potential $\widehat R_n$ in the form
given by \eqref{ar-hat} converges weakly to a non-degenerate process in $D(\rr_+,\rr).$ We remark that a somewhat shorter derivation of Theorem~\ref{sinai-thm} can be obtained by an adaptation of a version of Sinai's argument given in \cite[Section~2.5]{notes}. The version of the proof in this paper follows the approach of \cite{localization1} and is due to Dembo, Guionnet, and Zeitouni.
\section{Appendix}
\label{appen}
\subsection{Proof of Corollary~\ref{cor-1}}
\label{ape}
{\bf (a)} By part~(b) of Theorem~\ref{theorem-a},
\beq
E_Q(d_0)=\int d_0 \cdot \frac{d_0}{E_P(d_0)}dP=\frac{E_P(d_0^2)}{E_P(d_0)},
\feq
which is equivalent to the the first claim in the corollary.
$\mbox{}$
\\
{\bf (b)} By definition, $\witi a_0=\lfloor U d_0\rfloor,$ where $U$ is independent of $d_0$ under $Q.$ We therefore have,
\beq
E_Q(\witi a_0)&=&E_Q\bigl(E_Q(\witi a_0|d_0)\bigr)=E_Q\Bigl(\frac{1}{d_0}\sum_{k=0}^{d_0-1} j\Bigr) =\frac{1}{2}\bigl(E_Q(d_0)-1\bigr),
\feq
which verifies the second claim in the corollary. \qed
\subsection{Proof of Proposition~\ref{compare-speed}}
\label{ape1}
Recall $\ol S$ from \eqref{rols} and $\witi S$ from \eqref{witis}. First, notice that
\beq
\witi S=\sum_{i=1}^\infty \prod_{j=0}^{i-1}\witi \rho_{-j}+\sum_{i=1}^\infty \prod_{j=0}^i \witi \rho_{-j}+1+\witi \rho_0
=1+2\sum_{i=0}^\infty \prod_{j=0}^i \witi \rho_{-j}.
\feq
Next, observe that the following version of the identity \eqref{curious} can be stated in terms of the ``tilde environment" $(\witi a_n,\lambda_n)_{n\in\zz}:$
\beqn
\label{acurious}
\sum_{n=0}^{\infty}\witi \rho_0\witi \rho_1\cdots \witi \rho_n=\witi a_0+\sum_{n=0}^{\infty}\xi_0\xi_1\cdots \xi_n \cdot d_{n+1}.
\feqn
Therefore, using the stationarity of the environment under $Q$ along with the identity \eqref{acurious} and part~(b) of Corollary~\ref{cor-1},
\beq
E_Q\bigl(\witi S\bigr)&=&1+2E_Q\Bigl(\sum_{i=0}^\infty \prod_{j=0}^i \witi \rho_i\Bigr)
=E_Q(2\witi a_0+1)+ E_P(d_0)\cdot \bigl[E_Q(\ol S)-1\bigr]
\\
&=&
E_Q(d_0)+E_P(d_0)\cdot \bigl[E_Q(\ol S)-1\bigr]=\frac{E_P(d_0^2)}{E_P(d_0)}+ E_P(d_0)\cdot \bigl[E_Q(\ol S)-1\bigr]
\\
&=&
\frac{\mbox{\rm VAR}_P(d_0)}{E_P(d_0)}+ E_P(d_0)\cdot E_P(\ol S),
\feq
which together with Theorem~\ref{rec-thm} imply that the asymptotic speed of the RWSRE on the event $\{\lim_{n\to\infty} X_n=+\infty\}$ is $1\bigl\slash E_Q\bigl(\witi S\bigr),$ $\pp-\as,$
under the conditions of Theorem~\ref{lln-thm}. \qed

\subsection*{Acknowledgements}
The research of Anastasios Matzavinos is supported in part by the National Science Foundation under Grants NSF CDS\&E-MSS 1521266 and NSF CAREER 1552903. The research of Alexander Roitershtein is supported in part by the Simons Foundation under Collaboration Grant \#359575.  Alexander Roitershtein would also like to thank the Division of Applied Mathematics at Brown University for the warm hospitality during a visit, which was funded by a divisional IBM fund.

\bibliographystyle{amsplain}

\end{document}